\documentclass[11pt,twoside]{amsart}
\usepackage{amsmath}
\usepackage{amssymb}
\usepackage{amsthm}
\usepackage{latexsym}
\usepackage{amscd}
\usepackage{xypic}
\xyoption{curve}
\usepackage{ifthen} 
\usepackage{hyperref} 
\usepackage{graphicx}

\usepackage{xcolor}

\newtheorem{theorem}{Theorem}[section]
\newtheorem{lemma}[theorem]{Lemma}
\newtheorem{proposition}[theorem]{Proposition}

\newtheorem{remark}[theorem]{Remark}

\theoremstyle{definition}
\newtheorem{definition}[theorem]{Definition}

\newcommand {\coker}{\mathrm{coker}}
\newcommand {\gr}{\mathrm{gr}}
\newcommand {\spec}{\mathrm{spec}}

\newcommand {\Hom}{\mathcal{H}\kern -0.25ex{\mathit om\/}}
\newcommand {\Ext}{\mathcal{E}\kern -0.25ex{\mathit xt\/}}
\newcommand {\ext}{\mathrm{Ext}}
\newcommand {\im}{\mathrm{im}}
\newcommand {\rk}{\mathrm{rk}}

\newcommand {\End}{\mathrm{End}}
\newcommand {\Hilb}{\mathcal{H}\kern -0.25ex{\mathit ilb\/}}
\newcommand {\PGL}{\mathrm{PGL}}

\newcommand {\fC}{\mathfrak{C}}
\newcommand {\fE}{\mathfrak{E}}
\newcommand {\fX}{\mathfrak{X}}

\newcommand {\bZ}{\mathbb{Z}}

\newcommand {\bP}{\mathbb{P}}

\newcommand{\cE}{{\mathcal E}}
\newcommand{\cF}{{\mathcal F}}
\newcommand{\cM}{{\mathcal M}}
\newcommand{\cN}{{\mathcal N}}
\newcommand{\cO}{{\mathcal O}}
\newcommand{\cK}{{\mathcal K}}
\newcommand{\cQ}{{\mathcal Q}}
\newcommand{\cI}{{\mathcal I}}
\newcommand{\calH}{{\mathcal H}}
\newcommand{\calL}{{\mathcal L}}

\newcommand{\Pic}{\operatorname{Pic}}

\def\p#1{{\bP^{#1}}}

\def\ga#1{{{\accent"12 #1}}}
 
\title[Moduli spaces of rank two aCM vector bundles]{MODULI SPACES OF RANK TWO ACM BUNDLES
ON THE SEGRE PRODUCT OF THREE PROJECTIVE LINES}



\author[G. Casnati, D. Faenzi, F. Malaspina]{Gianfranco Casnati, Daniele Faenzi, Francesco Malaspina}
\thanks{All the authors are members of GRIFGA--GDRE project, supported by CNRS
  and INdAM, and of the GNSAGA group of INdAM. The first and third authors are supported by the framework of PRIN 2010/11 \lq Geometria delle variet\ga a algebriche\rq, cofinanced by MIUR. The second author is partially supported by ANR GEOLMI contract ANR-11-BS03-0011}

\begin{document}

\sloppy

\begin{abstract}
Let $F\subseteq\p7$ be the image of the Segre embedding of $\p1\times\p1\times\p1$. In the present paper we deal with the moduli spaces of locally free sheaves $\cE$ of rank $2$ with  $h^i\big(F,\cE(t)\big)=0$ for $i=1,2$ and $t\in\bZ$, on $F$.
\end{abstract}

\maketitle

\section{Introduction}
Let $\p N$ be the projective space of dimension $N$ over an algebraically closed field $k$ of characteristic $0$. If $F\subseteq\p N$ is an $n$--dimensional projective variety, i.e. an integral connected closed subscheme, we set $\cO_F(h):=\cO_{\p N}(1)\otimes\cO_F$.
We say that $F$ is arithmetically Cohen--Macaulay (aCM for short) if the natural restriction maps $H^0\big(\p
  N,\cO_{\p N}(t)\big)\to H^0\big(F,\cO_{F}(th)\big)$ are surjective
  and $H^i\big(F,\cO_{F}(th)\big)=0$, $1\le i\le n-1$. 
A vector bundle (i.e., a locally free sheaf) $\cE$ over such an $F$ is called aCM if all
the intermediate cohomology groups of $\cE$ vanishes, namely if $H^i\big(F,\cE(t h)\big)=0$ for $0<i<n$ and $t\in\bZ$.
 
If $F$ is just $\p n$, then a
well--known theorem of Horrocks (see \cite{O--S--S} and references
therein) states $\cE$ is aCM if and only if $\cE$ splits as direct sum of invertible sheaves.

When $F$ is a smooth quadric hypersurface Kn\"orrer's theorem
  (see \cite{Kno}) asserts that an indecomposable
aCM bundle $\cE$ on $F$ is either  $\cO_F$ or a spinor bundle, up to twists by multiples of $\cO_F(h)$ (see \cite{Ot} for the definition of spinor bundles on $F$ and its properties).

The case of hypersurfaces of higher degree is very interesting.
Indeed, an important theorem of \cite{B--G--S}
states that such an $F$ supports infinitely many
isomorphism classes of indecomposable aCM bundles.
These families have been studied by many authors: see for instance
  \cite{C--H1}.

Another interesting direction is to look at Fano varieties
i.e. smooth varieties such that the anticanonical sheaf $\omega_F^{-1}$ is
ample (see \cite{I--P} for a review about Fano varieties). The greatest positive integer $r$ such that
$\omega_F\cong\mathcal L^{-r}$ for some ample $\mathcal L\in\Pic(F)$ is called
the index of $F$. It is known that $1\le r\le n+1$ and
$r=n+1$ (resp. $r=n$) if and only if $F=\p n$ (resp. $F$ is a
smooth quadric hypersurface). This case is settled by the theorem of Horrocks (resp. Kn\"orrer).

Let us look at the next case
$r=n-1$. In this case $F$ is called a del Pezzo variety. Let $\mathcal L$ be very ample on $F$ and consider the corresponding embedding $F\subseteq\p{N}$.  Then $3\le \deg(F)\le 8$ and we know that  such an $F$ is also of \lq\lq almost minimal degree\rq\rq. Indeed
$\deg(F)=N-n+2$.

According to Eisenbud--Herzog classification theorem (see \cite{E--H}) $n$--dimensional non--degenerate subvarieties of $\p N$ supporting
only finitely many indecomposable aCM bundles (up to twist and
isomorphism) all have minimal degree $N-n+1$ (although not all varieties of
minimal degree have this property, see \cite{F--M} for a detailed treatment).
So again del Pezzo manifolds seem to be one of the most interesting
benchmarks to study aCM bundles.
Some results on vector bundles on del Pezzo surfaces are known (e.g. see \cite{C--K--M}, \cite{Fa1}). We focus our attention on the case $n=3$,
i.e. the case of threefolds. The first 
non--trivial bundles appear in rank two and we will particularly study this case.

When the Picard number $\varrho(F)$ is $1$, a complete classification of indecomposable aCM
bundles of rank $2$ on $F$ has been given by E. Arrondo and L. Costa
(see \cite{A--C}) using the so--called Hartshorne--Serre
correspondence between vector bundles of rank $2$ and subvarieties of
codimension $2$ satisfying an extra technical condition (see
\cite{Vo}, \cite{Ha1}, \cite{Ar} for details on such a
construction). More precisely they showed that if one twists such a
bundle $\cE$ by $\cO_F(th)$ in such a way that
$h^0\big(F,\cE\big)\ne0$ and $h^0\big(F,\cE(-h)\big)=0$ (we briefly
say that $\cE$ is initialized) and $c_1(\cE)=c_1h$, then $0\le c_1\le
2$ and it is possible to characterize $\cE$ in terms of the
zero--locus of a general section in $H^0\big(F,\cE\big)$. 

{It is also natural to analyze aCM bundles in terms of semistability and $\mu$--semistability (see \cite{H--L} as a reference for semistable bundles and their moduli spaces).
For del Pezzo threefolds with $\varrho(F)=1$, it is possible to show the following facts.
\begin{itemize}
\item If $c_1=0$, then $\cE$ is never semistable (though $\mu$--semistable).
\item If $c_1=1$, then $\cE$ is stable: M Szurek and J. Wi\'sniewski proved in \cite{S--W} that the corresponding moduli space is an irreducible projective variety of dimension $5-d$.
\item If $c_1=2$, then again $\cE$ is stable: moreover the corresponding moduli space was proved to be irreducible and of dimension $5$ by S. Druel when $d=3$ in \cite{Dr}, by A. Kuznetsov and by D. Faenzi independently when $d=4, 5$ in \cite{Kut} and \cite{Fa2} respectively.
\end{itemize}

When $\varrho(F)>1$ the only known results are due to the authors of the present paper when $\varrho(F)=3$ (see \cite{C--F--M}). In this case $F$ is exactly the image inside $\p7$ of the Segre embedding of $\p1\times\p1\times\p1$. Denote by $\pi_i\colon F\to \p1$ the $i^{\mathrm {th}}$--projection and let $\cO_F(h_i):=\pi_i^*\cO_{\p1}(1)$: then the intersection ring $A(F)$ of $F$ is isomorphic to $A(\p1)\otimes A(\p1)\otimes A(\p1)$  (see
\cite{Fu}, Example 8.3.7). In particular
$$
A(F)\cong\bZ[h_1,h_2,h_3]/(h_1^2,h_2^2,h_3^2).
$$
Now let $\cE$ be  an indecomposable, initialized, aCM bundle of rank $2$ on $F$ and set $c_1(\cE)=\alpha_1h_1+\alpha_2h_2+\alpha_3h_3$. In the aforementioned paper it is proved that, up to permutations of $h_i$'s, only the following cases are possible (and actually occur) for $(\alpha_1,\alpha_2,\alpha_3)$:
$$
(0,0,0),\qquad (0,0,1), \qquad (2,2,1),\qquad (1,2,3),\qquad (2,2,2).
$$
Notice that there exist initialized, aCM bundles $\cE$ of rank $2$ on $F$ with $\det(\cE)=\cO_F(h)$, but they are always decomposable as $\cO_F(h_1+h_2)\oplus\cO_F(h_3)$, up to permutations of the $h_i$'s, thus they are not $\mu$--semistable (or, in a more suggestive form inspired by the lower degree cases, the locus of such bundles has dimension $-1=5-6$).

The aim of the present paper is to construct and describe the moduli spaces of such bundles in the above cases, when they are semistable. We are able to prove the following statement in Section \ref{sRank}.

\medbreak
\noindent
{\bf Theorem A.}
{\it Let $\cE$ be an indecomposable, initialized, aCM bundle of rank $2$ on $F$ and let $c_1(\cE)=\alpha_1h_1+\alpha_2h_2+\alpha_3h_3$. 
Then the following assertions hold.
  \begin{enumerate}
  \item If $(\alpha_1,\alpha_2,\alpha_3)=(0,0,0)$, then $\cE$ is strictly $\mu$--semistable.
  \item If $(\alpha_1,\alpha_2,\alpha_3)$ is either $(0,0,1)$ or $(2,2,1)$, then $\cE$ is $\mu$--stable.
  \item If $(\alpha_1,\alpha_2,\alpha_3)=(2,2,2)$, then $\cE$ is $\mu$--stable, unless it fits into an exact sequence of the form
\begin{equation*}
0\longrightarrow\cO_F(2h_1+2h_2+h_3-2h_i)\longrightarrow \cE\longrightarrow\cO_F (2h_i+h_3)\longrightarrow0,
\end{equation*}
where $i=1,2$, in which case is strictly semistable.
  \item If $(\alpha_1,\alpha_2,\alpha_3)=(1,2,3)$, then $\cE$ is $\mu$--stable, unless it fits into an exact sequence of the form
\begin{equation*}
0\longrightarrow\cO_F(h_1+2h_3)\longrightarrow \cE\longrightarrow\cO_F (2h_2+h_3)\longrightarrow0,
\end{equation*}
in which case is strictly semistable.
  \end{enumerate}}
  \medbreak

In view of the above theorem, it is interesting to understand the structure of the moduli spaces corresponding to semistable bundles. To this purpose we first need to prove their existence. In \cite{C--F--M} also the second Chern class of the bundle is computed. We obtain the following results.
\begin{itemize}
\item If $(\alpha_1,\alpha_2,\alpha_3)=(0,0,1)$, then $c_2(\cE)$ is either $h_2h_3$, or $h_1h_3$.
\item If $(\alpha_1,\alpha_2,\alpha_3)=(2,2,1)$, then $c_2(\cE)$ is either $h_2h_3+2h_1h_3+2h_1h_2$, or $2h_2h_3+h_1h_3+2h_1h_2$.
\item If $(\alpha_1,\alpha_2,\alpha_3)=(1,2,3)$, then either $c_2(\cE)=4h_2h_3+h_1h_3+2h_1h_2$ or $c_2(\cE)=3h_2h_3+3h_1h_3+h_1h_2$.
\item If $(\alpha_1,\alpha_2,\alpha_3)=(2,2,2)$, then, up to permutations of the $h_i$'s, either $c_2(\cE)=2h_2h_3+3h_1h_3+3h_1h_2$ or $c_2(\cE)=2h_2h_3+2h_1h_3+4h_1h_2$.
\end{itemize}

We first construct,  in Section \ref{sHilbert}, the Hilbert schemes of curves inside $F$ associated to such bundles via the aforementioned Harshorne--Serre correspondence. Then we show how to define a universal family parameterizing such bundles in Section \ref{sRelative} by using a relative version of the Hartshorne--Serre correspondence.

Finally, in Sections \ref{sModuliUlrich},  \ref{sModuliLine} and \ref{sRational}, we use such a family to show the irreducibility and unirationality of the corresponding moduli spaces. We can roughly summarize what we are able to prove in the following statement.

\medbreak
\noindent
{\bf Theorem B.}
{\it Let $c_1:=\alpha_1h_1+\alpha_2h_2+\alpha_3h_3$ and $c_2:=\beta_1h_2h_3+\beta_2h_1h_3+\beta_3h_1h_2$. If $(\alpha_1,\alpha_2,\alpha_3)$ is one of the following
$$
(0,0,1), \qquad (2,2,1),\qquad (1,2,3),\qquad (2,2,2),
$$
then the moduli space $\mathcal M(c_1,c_2)$  of indecomposable, initialized, aCM semistable bundles $\cE$ of rank $2$ with $c_1(\cE)=c_1$ and $c_2(\cE)=c_2$ exists and it is irreducible. Moreover the following assertions hold.
\begin{enumerate}
\item Let $(\alpha_1,\alpha_2,\alpha_3)$ be either $(0,0,1)$ or $(2,2,1)$. 
Then $\mathcal M(c_1,c_2)\cong\p1$
\item Let $(\alpha_1,\alpha_2,\alpha_3)=(1,2,3)$. 

The moduli space $\mathcal M(c_1,4h_2h_3+h_1h_3+2h_1h_2)$ is a single point, representing the equivalence class of all the strictly semistable bundles with such a $c_1$. 

The moduli space $\mathcal M(c_1,3h_2h_3+3h_1h_3+h_1h_2)$ is smooth and unirational of dimension $3$: its  points correspond to stable bundles.
\item Let $(\alpha_1,\alpha_2,\alpha_3)=(2,2,2)$. 

The moduli space $\mathcal M(c_1,2h_2h_3+2h_1h_3+4h_1h_2)$ is generically smooth and rational of dimension $5$: its general point corresponds to a stable bundle and it also contains exactly one point representing the equivalence class of all the strictly semistable bundles with such a $c_1$.  

The moduli space $\mathcal M(c_1,2h_2h_3+3h_1h_3+3h_1h_2)$ is smooth and unirational of dimension $5$: its  points correspond to stable bundles.
\end{enumerate}}
\medbreak

\section{aCM and semistable bundles on $F$}
\label{sGeneral}
If $\cE$ is an aCM bundle, then the minimal number of
generators of $H^0_*\big(F,\cE\big)$ as a module over the
graded coordinate ring of $F$ is $\rk(\cE)\deg(F)$ at
most (e.g. see \cite{C--H1}). For the above reason we introduce the following definition (see  \cite{C--H2}, Definition 2.1 and Lemma 2.2: see also \cite{C--H1}, Definition 3.4
which is slightly weaker).

\begin{definition}
  Let $\cE$ be a vector bundle on $F$.
  We say that $\cE$ is {\sl initialized} if 
  $$
  \min\{\ t\in\bZ\ \vert\ h^0\big(F,\cE(t h)\big)\ne0\ \}=0.
  $$
  We say that $\cE$ is {\sl Ulrich} if it is initialized, aCM and
  $h^0\big(F,\cE\big)=\rk(\cE)\deg(F)$. 
\end{definition}

Notice that Ulrich bundles are globally generated by definition.

Let us now recall some notions of stability of vector bundles.
First we define the slope $\mu(\cE)$ and the reduced Hilbert
polynomial $p_{\cE}(t)$ of a bundle $\cE$ over $F$ as:
$$
\mu(\cE)= c_1(\mathcal E)h^2/\rk(\mathcal E), \qquad p_{\cE}(t)=\chi(\cE(th))/\rk(\cE).
$$
The bundle $\cE$ is called $\mu$--semistable (with respect to
$h$) if for all non--zero torsion--free proper quotient bundles
$\mathcal G$ we have
$$
\mu(\mathcal G) \ge \mu(\cE),
$$
and $\mu$--stable if equality cannot hold in the above inequality.
On the other hand, $\cE$ is said to be semistable (or, more precisely,
Gieseker--semistable with respect to $h$) if for all $\mathcal G$ as above one has
$$
p_{\mathcal G}(t) \ge  p_{\cE}(t),
$$
and (Gieseker) stable again if equality cannot hold in the above inequality.

Let $\cE$ be a vector bundle on $F$ of rank $r$ with Hilbert polynomial $\chi(t):=\chi(\mathcal E(th))$.
Recall that there exists the coarse moduli spaces $\cM_F^{ss}(\chi)$ parameterizing $S$--equivalence classes of semistable rank $r$ torsion free coherent sheaves on $F$ with Hilbert polynomial $\chi(t)$ (see Section 1.5 of \cite{H--L} for details about $S$--equivalence of bundles). We will denote by $\cM_F^{s}(\chi)$ the open locus inside $\cM_F^{ss}(\chi)$ of stable bundles.

The scheme $\cM_F^{ss}(\chi)$ is the disjoint union of open and closed subsets $\cM_F^{ss}(r;c_1,\dots,c_r)$ whose points represent $S$-equivalence classes of semistable rank $r$ torsion free coherent sheaves with fixed Chern classes $c_i\in A^i(F)$ ($A^i(F)$ denotes the degree $i^{\mathrm {th}}$ component of the intersection ring $A(F)$ of $F$). Similarly $\cM_F^{s}(\chi)$ is the disjoint union of open and closed subsets $\cM_F^{s}(r;c_1,\dots,c_r)$.

By semicontinuity we can define open loci $\cM_F^{ss,aCM}(\chi)(r;c_1,\dots,c_r)\subseteq \cM_F^{ss}(\chi)(r;c_1,\dots,c_r)$ and $\cM_F^{s,aCM}(\chi)(r;c_1,\dots,c_r)\subseteq \cM_F^{s}(\chi)(r;c_1,\dots,c_r)$ parameterizing respectively $S$--equivalence classes of semistable and stable aCM bundles of rank $r$ on $F$ with Chern classes $c_1,\dots,c_r$.

The case of Ulrich bundles is particularly interesting. Indeed they are globally generated by definition and semistable  (see  \cite{C--H2}, Theorem 2.9), hence $\mu$--semistable. Moreover their reduced Hilbert polynomial is
$$
p_{\cE}(t):=\chi(\cE(th))/\rk(\cE)=\deg(F){{t+3}\choose3}
$$ 
(e.g., see \cite{C--H2}, Lemma 2.6).

For the following proposition see \cite{C--H2}.

\begin{proposition}
\label{pModUlrich}
There exist coarse moduli spaces $\cM_F^{ss,U}(r;c_1,\dots,c_r)$ and $\cM_F^{s,U}(r;c_1,\dots,c_r)$ for respectively semistable and stable Ulrich bundles of rank $r$ on $F$ with Chern classes $c_1,\dots,c_r$.
\end{proposition}

A helpful result about Ulrich bundles is the following.

\begin{lemma}
\label{lObstructed} 
Let $F$ be a del Pezzo threefold. If $\cE$ is an Ulrich bundle of rank $r$ on $F$, then 
\begin{gather*}
h^2\big(F,\cE\otimes {\cE}^\vee(th)\big)=0,\qquad t\ge0,\\
h^3\big(F,\cE\otimes {\cE}^\vee(th)\big)=0,\qquad t\ge-1.
\end{gather*}
In particular, stable Ulrich bundles  with Chern classes $c_1,\dots,c_r$, if any, correspond to smooth points of $\cM_F^{s,U}(r;c_1,\dots,c_r)$.
\end{lemma}
\begin{proof}
If $\cE$ is Ulrich, then there exists a presentation of the form
$$
\cO_{\p7}(-1)^{\oplus\beta_1}\longrightarrow\cO_{\p7}^{\oplus\beta_0}\to\cE\longrightarrow0
$$
(see \cite{E--S--W}). Twisting such a sequence by $\cO_F$ we obtain an exact sequence of the form
$$
\cO_{F}(-h)^{\oplus\beta_1}\longrightarrow\cO_{F}^{\oplus\beta_0}\to\cE\longrightarrow0.
$$
If we denote by $\mathcal K$ the image of $\cO_{F}(-h)^{\oplus\beta_1}\to\cO_{F}^{\oplus\beta_0}$, then we finally obtain the exact sequence 
$$
0\longrightarrow\mathcal K\longrightarrow\cO_{F}^{\oplus\beta_0}\to\cE\longrightarrow0.
$$
The sheaf $\mathcal K$ is locally free on $F$, because the same is true for both $\cO_{F}^{\oplus\beta_0}$ and $\cE$.

Twisting such a sequence by ${\cE}^\vee(th)$ and taking its cohomology, we obtain
\begin{gather*}
h^3\big(F,\cE\otimes {\cE}^\vee(th)\big)\le \beta_0h^3\big(F,\cE^\vee(th)\big)=\beta_0h^0\big(F,\cE((-t-2)h)\big),\\
h^2\big(F,\cE\otimes {\cE}^\vee(th)\big)\le h^3\big(F,{\mathcal K}\otimes {\cE}^\vee(th)\big)=h^0\big(F,{\mathcal K}^\vee\otimes\cE((-t-2)h)\big),
\end{gather*}
because $\cE$ is aCM and $F$ has dimension $3$. 
If $t\ge-1$, then $h^0\big(F,\cE((-t-2)h)\big)$ because $\cE$ is initialized. Thus $h^3\big(F,\cE\otimes {\cE}^\vee(th)\big)=0$ in such a range.

The epimorphism $\cO_{F}(-1)^{\oplus\beta_1}\twoheadrightarrow\mathcal K$ induces by duality a monomorphism ${\mathcal K}^\vee\otimes\cE((-t-2)h)\rightarrowtail\cE((-t-1)h)^{\oplus\beta_1}$. Thus 
$$
h^0\big(F,{\mathcal K}^\vee\otimes\cE((-t-2)h)\big)\le \beta_1h^0\big(F,\cE((-t-1)h)\big)=0
$$
if $t\ge0$, because $\cE$ is initialized by definition.
\end{proof}

Now assume that $\cE$ has rank $2$. If $s\in H^0\big(F,\cE$\big), then its zero--locus
$(s)_0\subseteq F$ is either empty or its codimension is at most
$2$.
 Assume that we are in the second case and that the codimension is actually $2$. Thus we can consider its Koszul complex 
\begin{equation}
  \label{seqIdeal}
  0\longrightarrow \cO_F\longrightarrow \cE\longrightarrow \cI_{C\vert F}(c_1)\longrightarrow 0,
\end{equation}
where $\cI_{C \vert F}$ denotes the sheaf of ideals of $C:=(s)_0$ inside $F$.
Moreover we also have the following exact {sequence}
\begin{equation}
  \label{seqStandard}
  0\longrightarrow \cI_{E\vert F}\longrightarrow \cO_F\longrightarrow \cO_E\longrightarrow 0.
\end{equation}

The above construction can be reversed, giving rise to Hartshorne--Serre correspondence (for further details about the statement in the general case see \cite{Vo}, \cite{Ha1}, \cite{Ar}). We will inspect a relative form of such a correspondence later on in Section \ref{sRelative}.

\section{Semistability of aCM bundles on $F$ of rank $2$}
\label{sRank}
Now we restrict our attention to rank $2$ aCM vector bundles on the del Pezzo threefold $F:=\p1\times\p1\times\p1\subseteq\p7$.

We proved in \cite{C--F--M} that each general section of such a bundle vanishes exactly along a curve. Moreover, making use of this fact, we also classified therein all such bundles, obtaining the following proposition listing all the possible cases.

\begin{proposition}
\label{pList}
Let $\cE$ be an indecomposable, initialized, aCM vector bundle of rank $2$ on $F:=\p1\times\p1\times\p1$ with $c_1:=c_1(\cE)$ and $c_2(\cE):=c_2$. Let $C$ be the zero--locus of a general section of $\cE$ and denote by $p_C(t)$ its Hilbert polynomial. Then the following possibilities hold for $\cE$, up permutations of the $h_i$'s.
\begin{enumerate}
\item $\cE$ satisfies either $c_1=0$ or $c_1=h_3$: then we can assume $c_2=h_2h_3$. We have $p_C(t)=t+1$ and $C$ is a line, thus it is irreducible. Moreover, each such curve can be obtained in this way.
\item $\cE$ satisfies $c_1=2h_1+2h_2+h_3$: then we can assume  $c_2=h_2h_3+2h_1h_3+2h_1h_2$. We have $p_C(t)=5t+1$ and $C$ is a possibly reducible quintic curve.
\item $\cE$ satisfies $c_1=h_1+2h_2+3h_3$: then either $c_2=4h_2h_3+h_1h_3+2h_1h_2$ or $c_2=3h_2h_3+3h_1h_3+h_1h_2$. In this case $\cE$ is Ulrich, hence globally generated and $p_C(t)=7t+1$, thus $C$ can be assumed to be a rational normal curve (in particular its embedding $C\subseteq\p7$ is non--degenerate). Moreover, each such curve can be obtained in this way.
\item $\cE$ satisfies $c_1=2h$: then we can assume either $c_2=2h_2h_3+3h_1h_3+3h_1h_2$ or $c_2=2h_2h_3+2h_1h_3+4h_1h_2$. In this case $\cE$ is Ulrich, hence globally generated and $p_C(t)=8t$, thus $C$ can be assumed to be an elliptic normal curve (in particular its embedding $C\subseteq\p7$ is non--degenerate). Moreover, each such curve can be obtained in this way.
\end{enumerate}
\end{proposition}

We call the above Chern classes {\sl representative }\/.

We are interested in dealing with moduli spaces of rank $2$ aCM semistable bundles on $F$. Thus the very first step in our study is to check whether such semistable bundles actually exist.

Assume $c_1$ is either $2h$ or $h_1+2h_2+3h_3$. In this case $\cE$ is Ulrich. Theorem 2.9 (c) of \cite{C--H2} shows that $\cE$ is stable if and only if it is $\mu$--stable. It is interesting to find a simple condition which guarantees the stability or strict semistability of such bundles. 

\begin{proposition}
\label{pStable}
Let $\cE$ be an Ulrich bundle of rank $2$ on $F$. 

The vector bundle $\cE$ is a strictly semistable Ulrich bundle if and only if, up to permutations of the $h_i$'s, it fits into an exact sequence of the form
\begin{equation}
\label{seqUnstable}
0\longrightarrow\calL\longrightarrow \cE\longrightarrow\cO_F (2h_2+h_3)\longrightarrow0,
\end{equation}
where $\calL$ is either $\cO_F(2h_1+h_3)$ or $\cO_F(h_1+2h_3)$.
\end{proposition}
\begin{proof}
Assume that $\cE$ is an Ulrich bundle. 

Let $c_1=2h$: hence $c_2$ can be assumed to be either $2h_2h_3+2h_1h_3+4h_1h_2$ or $2h_2h_3+3h_1h_3+3h_1h_2$. 

We already know that $\cE$ is semistable, whence $\mu$--semistable. Assume it is not  $\mu$--stable. It follows the existence of sheaves $\calL$ and $\cM$ of rank $1$ with $\mu(\cM)=\mu(\cE)=6$, $\cM$ torsion free, fitting into a sequence of the form
$$
0\longrightarrow\calL\longrightarrow \cE\longrightarrow\cM\longrightarrow0.
$$
By the additivity of the first Chern class we obtain that $\mu(\calL)=6$. Hence Theorem 2.8 of \cite{C--H2} implies that $\calL$ and $\cM$ are both Ulrich bundles on $F$. In \cite{C--F--M} the complete list of Ulrich invertible sheaves is given. Taking into account that $\mu(\cM)=6$, it follows that $\cM\cong\cO_F(\alpha_1h_1+\alpha_2h_2+\alpha_3h_3)$ where $(\alpha_1,\alpha_2,\alpha_3)$ is, up to permutations, $(0,1,2)$. Consequently $\calL\cong\cO_F((2-\alpha_1)h_1+(2-\alpha_2)h_2+(2-\alpha_3)h_3)$.

Computing $c_2$ from the exact sequences we deduce that  $\calL$ and $\cM$ are either $\cO_F(2h_2+h_3)$ and $\cO_F(2h_1+h_3)$, or $\cO_F(2h_1+h_3)$ and $\cO_F(2h_2+h_3)$. We have thus proved the existence of Sequence \eqref{seqUnstable} in the case $c_1=2h$.

When $c_1=h_1+2h_2+3h_3$, then $c_2$ is either $4h_2h_3+h_1h_3+2h_1h_2$ or $c_2=3h_2h_3+3h_1h_3+h_1h_2$. The same argument used in the case $c_1=2h$ shows that the only possibility for $\cE$ to be semistable is that it fits in an exact sequence of the form 
$$
0\longrightarrow\cO_F(h_1+2h_3)\longrightarrow \cE\longrightarrow\cO_F (2h_2+h_3)\longrightarrow0,
$$

Conversely assume that $\cE$ fits in Sequence \eqref{seqUnstable}. On the one hand it is immediate to check that 
$$
\mu(\cO_F(2h_2+h_3))=\mu(\cO_F(2h_1+h_3))=\mu(\cO_F(h_1+2h_3))=6=\mu(\cE), 
$$
thus $\cE$ is never $\mu$--stable in the above cases. On the other hand easy computations show that $\cE$ is Ulrich, thus $\cE$ is semistable.
\end{proof}

Let us consider initialized, aCM vector bundles associated to lines on $F$: for such bundles we can assume that $c_1$ is either $0$ or $h_3$. We start with the case $c_1=0$.

\begin{proposition}
\label{pLine}
Let $\cE$ be an aCM vector bundle of rank $2$ on $F$ with $c_1=0$. Then $\cE$ is $\mu$--semistable, but not semistable.
\end{proposition}
\begin{proof}
Let $C$ be the zero--locus of a general section of $\cE$ corresponding to the subbundle $\cO_F\subseteq\cE$, hence to the torsion free quotient $\cI_{C\vert F}\cong\cE/\cO_F$.
Sequences \eqref{seqIdeal} and \eqref{seqStandard} for $C$ give
$$
p_{\cE}(t)=\frac{\chi(\cE(th))}{2}=\chi(\cI_{C\vert F}(th))+{\frac12}\chi(\cO_C(th))=p_{\cI_{C\vert F}}(t)+{\frac12}(t+1).
$$
It follows that $\cE$ is not semistable.

Now we prove that $\cE$ is $\mu$--semistable. If not there should exist a torsion--free quotient sheaf $\cQ$ of $\cE$ of rank $1$ such that $\mu(\cQ)<\mu(\cE)=0$. Being $\cQ$ torsion--free, then the canonical morphism to the bidual of $\cQ$   is injective. The bidual, being reflexive, is an invertible sheaf on $F$ (see Lemma II.1.1.15 of \cite{O--S--S}), say $\cO_F(q_1)$ with $q_1:=c_1(\cQ)$, so that $q_1h^2=\mu(\cQ)<0$, thus $\cQ=\cI_{S\vert F}(q_1)$ where $S$ has codimension at least $2$. The kernel $\cK$ of the quotient morphism $\cE\twoheadrightarrow\cQ$ is torsion--free, normal (see \cite{O--S--S}, Lemma II.1.1.16) and of rank $1$, thus it is invertible (see Lemmas II.1.1.12 and II.1.1.15 of \cite{O--S--S}). The additivity of the first Chern class thus implies $\cK\cong\cO_F(-q_1)\not\cong\cO_F$.

Again, let $C$ be the zero locus of a general section of $\cE$. The corresponding inclusion $\cO_F\subseteq\cE$ induces by composition a morphism $\cO_F\to\cQ\cong \cI_{S\vert F}(q_1)$. Such a map must be zero, because, otherwise, there would be a divisor of degree $\mu(\cQ)<0$ through $S$.

We deduce that the non--zero morphism $\cO_F\to\cE$ factors through an inclusion $\cO_F\subseteq\cO_F(-q_1)$. In particular we have a commutative diagram
$$
\begin{CD}
0@>>>\cO_F@>>>\cE@>>>\cI_{C\vert F}@>>>0\\
@.        @VfVV            @|          @VgVV                 @.\\
0@>>>\cO_F(-q_1)@>>>\cE@>>>\cI_{S\vert F}(q_1)@>>>0.
\end{CD}
$$
Snake's Lemma yields $\coker(f)\cong\ker(g)\subseteq\cI_{C\vert F}$. In particular $\coker(f)$ is torsion--free. Since both $\cO_F$ and $\cO_F(-q_1)$ are invertible sheaves it follows that $\coker(f)=0$, whence $\cO_F(-q_1)\cong\cO_F$, contradicting the inequality $q_1h^2<0$ proved above. 

The contradiction proves the statement.
\end{proof}

Now we focus our attention to the other kind of initialized, aCM vector bundles $\cE$ associated to lines. We will check that they are $\mu$--stable.

We assume $c_1=h_3$. For dealing with the $\mu$--stability of such bundles, one can repeat the argument used in the previous proposition for proving the $\mu$--semistability almost word by word. Indeed, we still take a torsion--free quotient $\cQ$ of $\cE$, but, in this case, we must assume $\mu(\cQ)\le\mu(\cE)=1$. We notice that the kernel $\cK$ of the quotient morphism $\cE\twoheadrightarrow\cQ$ satisfies $\mu(\cK)=2-\mu(\cQ)\ge1$. Since
$$
\mu(\cK)=c_1(\cK)h^2=2c_1(\cK)(h_2h_3+h_1h_3+h_1h_2),
$$
is even we infer that $\mu(\cK)\ge2$, whence again $\mu(\cQ)=2-\mu(\cK)\le0$. Let $C$ be the zero locus of a general section of $\cE$ corresponding to the inclusion $\cO_F\subseteq\cE$. Thus we again have a morphism $\cO_F\to\cQ$ which must be zero, because, otherwise, there would be a divisor of degree $\mu(\cQ)\le0$ through $S$. 

At this point, along the same lines of the proof of Proposition \ref{pLine}, we obtain the existence of $\cE\twoheadrightarrow\cO_F$. It would follow that $h^0\big(F,\cE(-h_2)\big)=h^0\big(F,{\cE}^\vee\big)\ne0$. If $C$ is the zero locus of a general section of $\cE$, then the cohomology of Sequence \eqref{seqIdeal} twisted by $\cO_F(-h_2)$ finally yields $h^0\big(F,\cE(-h_2)\big)\ne0$, a contradiction.

Now let $\cE$ be an initialized, aCM, vector bundle with $c_1=2h_1+2h_2+h_3$. On the one hand this occurs if and only if ${\cE}^\vee(h)$ is an initialized aCM vector bundle with $c_1=h_3$. On the other hand we know that $\cE$ is $\mu$--stable if and only if the same is true for ${\cE}^\vee(h)$. Thanks to the above analysis we have completed the proof of the following result.

\begin{proposition}
\label{pOther}
Let $\cE$ be an aCM vector bundle of rank $2$ on $F$ with $c_1$ either $h_3$ or $2h_1+2h_2+h_3$ up to permutations of the the $h_i$'s. Then $\cE$ is $\mu$--stable.
\end{proposition}

We conclude that the cases we are interested in are when $c_1$ is either $h_3$, or $2h_1+2h_2+h_3$, or $2h$, or $h_1+2h_2+3h_3$. In order to deal with the corresponding moduli spaces we first describe the Hilbert schemes of the corresponding associated curves.

\section{Hilbert schemes of curves on $F$}
\label{sHilbert}

In this section we will list and prove some results about Hilbert
schemes of curves on $F$ corresponding to some representative Chern
classes. 

Given a curve $C$ in $F$, the local structure of the Hilbert scheme around the point corresponding to $C$ is controlled by the normal sheaf $\cN_{C\vert F}$ of $C$ inside $F$, i.e. by the $\cO_F$--dual
of $\cI_{C\vert F}/\cI_{C\vert F}^2$.

We start with curves whose class in $A^2(F)$ is $h_2h_3$. Such curves are lines. The following result is partially well--known (see \cite{I--P}, Proposition 3.5.6).

\begin{proposition}
\label{pHilbertLine}
The scheme $\Hilb_{t+1}(F)$ has exactly three disjoint components. Each of them is the locus of points representing one and the same class inside $A^2(F)$ and it is isomorphic to $\p1\times\p1$. 
\end{proposition}
\begin{proof}
The Hilbert scheme $\Hilb_{t+1}(F)$ has exactly three components isomorphic to $\p1\times\p1$ (see \cite{I--P}, Proposition 3.5.6). 

Let $\calH$ one of them and consider the universal family $\mathcal C\to \calH$, i.e. the flat family whose fibre over a point is the corresponding line. Two lines in this family are algebraically equivalent, hence they are also homologically equivalent (see \cite{Fu}, Proposition 19.1.1). Thus any two fibres of $\mathcal C$ are actually rationally equivalent, because $F$ is homogeneous (see Example 19.1.11 of \cite{Fu}). Thus the Chern classes of the points of $\calH$ inside $A^2(F)$ are constant.
\end{proof}

Now we turn our attention to $\Hilb_{7t+1}(F)$. Let $\Hilb_{7t+1}^{sm}(F)$ be the open locus corresponding to smooth and connected curves and let $\Hilb_{7t+1}^{sm,nd}(F)$ be the subset corresponding to non--degenerate curves. Theorem 2.1 of \cite{C--G--N} implies that such a condition is equivalent to the fact that $C$ is aCM, which is an open condition on flat families because it corresponds to the vanishing of some cohomology groups. Thus $\Hilb_{7t+1}^{sm,nd}(F)$ is open too inside $\Hilb_{7t+1}(F)$.

As pointed out in Proposition \ref{pList}, we can restrict our attention to curves $C$ whose class in $A^2(F)$ is either $4h_2h_3+h_1h_3+2h_1h_2$, or $3h_2h_3+3h_1h_3+h_1h_2$. Indeed, in Section 7 of \cite{C--F--M}, we showed that if $C\in \Hilb_{7t+1}^{sm,nd}(F)$, then its class is either one of them, or $3h_2h_3+2h_1h_3+2h_1h_2$, up to permutations of the $h_i$'s and that all the above cases actually occur. Nevertheless, $C$ is the zero locus of a section of an aCM bundle $\cE$ only in the two former cases. In these cases $\cE$ is Ulrich and $c_1(\cE)=h_1+2h_2+3h_3$.

Since the Chern classes are fixed up to permutations of the $h_i$'s, we have exactly twelve possible cases.

\begin{proposition}
\label{pRational}
The scheme $\Hilb_{7t+1}^{sm,nd}(F)$ has exactly twelve disjoint components. Each of them is the locus of points representing one and the same class inside $A^2(F)$, is smooth, unirational and has dimension $14$.
\end{proposition}
\begin{proof}
We want to prove that the locus  $\calH_{c_2}\subseteq\Hilb_{7t+1}^{sm,nd}(F)$ of points representing curves whose class in $A^2(F)$ is $c_2$ is actually irreducible.
It suffices to prove the irreducibility of the locus $\overline{\calH}_{c_2}$ in $\Hilb_{7t+1}^{sm}(F)$ of, not necessarily skew, curves whose class is $c_2$: indeed $\calH_{c_2}$ is open inside $\overline{\calH}_{c_2}$ because it trivially coincides with $\overline{\calH}_{c_2}\cap\Hilb_{7t+1}^{sm,nd}(F)$. 

We will prove that $\overline{\calH}_{c_2}$ is dominated by an irreducible variety. To this purpose we first construct a scheme parameterizing maps from $\p1$ to $F$ such that the class of the image in $A^2(F)$ is fixed.

Fix the attention on $c_2:=3h_2h_3+3h_1h_3+h_1h_2$, the other case being similar. To give a  morphism $\alpha\colon \p1\to F$ such that the class $\deg(\alpha)\im(\alpha)$ in $A^2(F)$ is $c_2$ is the same as to give three pairs of linearly independent sections in $H^0\big( \p1,\cO_\p1(3)\big)$, $H^0\big( \p1,\cO_\p1(3)\big)$, $H^0\big( \p1,\cO_\p1(1)\big)$, thus a general element of 
$$
Y:=H^0\big( \p1,\cO_\p1(3)\big)^{\oplus2}\times H^0\big( \p1,\cO_\p1(3)\big)^{\oplus2}\times H^0\big( \p1,\cO_\p1(1)\big)^{\oplus2}.
$$
For a general choice of such an element the map $\alpha$ is an isomorphism onto its image. Let $Y_0\subseteq Y$ be the open and non--empty locus of points satisfying such a condition. We have a natural family ${\mathcal Y}_0\subseteq Y_0\times F$ whose fibres are smooth rational curves on $F$ of degree $7$, whence such a family is flat. The universal property of the Hilbert scheme yields the existence of a unique morphism $Y_0\to\Hilb_{7t+1}^{sm}(F)$ whose image is $\overline{\calH}_{c_2}$, which is thus irreducible. Since $Y_0$ is trivially a rational variety, it follows that $\overline{\calH}_{c_2}$ is also unirational.

Finally we have to prove that ${\calH}_{c_2}$ is smooth of dimension $14$. To this purpose we pick a point of ${\calH}_{c_2}$ corresponding to a smooth, connected, rational curve $C$ and we compute $h^0\big(F,\cN_{C\vert F}\big)$ and $h^1\big(F,\cN_{C\vert F}\big)$.  Taking into account that $C$ is rational, we know that $\cN_{C\vert F}\cong\cO_{\p1}(a)\oplus\cO_{\p1}(b)$ for suitable integers $a$ and $b$, thanks to a theorem of Grothendieck. 

By adjunction $\cO_\p1(-2)\cong\omega_C\cong\det(\cN_{C\vert F})\otimes\cO_F(-2h)$, thus $\det(\cN_{C\vert F})\cong\cO_{\p1}(12)$, hence $a+b=12$. 

Finally recall that there is a surjection $ \Omega^\vee_F\otimes\cO_C\twoheadrightarrow \cN_{C\vert F}$. Since $\Omega_F\cong\bigoplus_{i=1}^3\cO_{F}(-2h_i)$, it follows that $\cN_{C\vert F}$ is globally generated, thus $a,b\ge0$. We conclude that $h^0\big(F,\cN_{C\vert F}\big)=14$ and $h^1\big(F,\cN_{C\vert F}\big)=0$. 

Since each component ${\calH}_{c_2}$ is globally smooth, we also conclude that the components of $\Hilb_{7t+1}^{sm,nd}(F)$ are necessarily disjoint.
\end{proof}

We conclude with a similar analysis for elliptic curves. We will again denote by $\Hilb_{8t}^{sm,nd}(F)$ the locus inside $\Hilb_{8t}(F)$ of points representing non--degenerate, smooth and connected curves. As pointed out in Proposition \ref{pList} (see also Section 6 of \cite{C--F--M}), if $C\in \Hilb_{8t}^{sm,nd}(F)$, then its class is either $2h_2h_3+3h_1h_3+3h_1h_2$, or $2h_2h_3+2h_1h_3+4h_1h_2$, up to permutations of the $h_i$'s and that all these cases actually occur. Moreover $\cE$ is Ulrich (see \cite{C--H1} for the definition and properties of such bundles) and $c_1(\cE)=2h$. 

Since the Chern classes are fixed up to permutations of the $h_i$'s, we have exactly six possible cases. 

\begin{proposition}
\label{pElliptic}
The scheme $\Hilb_{8t}^{sm,nd}(F)$ has exactly six disjoint components. Each of them is the locus of points representing one and the same class inside $A^2(F)$, is smooth, unirational and has dimension $16$.
\end{proposition}
\begin{proof}
The proof runs along the same lines of the proof of Proposition \ref{pRational}. Again we can define $\Hilb_{8t}^{sm}(F)$ as the locus of smooth and connected elliptic curves of degree $8$ inside $F$. We will prove  that the locus  $\calH_{c_2}\subseteq\Hilb_{8t}^{sm,nd}(F)$ of points representing curves whose class in $A^2(F)$ is $c_2$ is actually irreducible by constructing an irreducible scheme parameterizing maps from elliptic curves to $F$ such that the class of the image in $A^2(F)$ is fixed.

Fix the attention on $c_2:=2h_2h_3+3h_1h_3+3h_1h_2$, the other case being similar. If $C$ is an elliptic curve, then to give a  morphism $\alpha\colon C\to F$ such that the class of $\deg(\alpha)\im(\alpha)$ in $A^2(F)$ is $c_2$ is the same as to give three points $p_1,p_2,p_3\in C$ and three pairs of linearly independent sections in $H^0\big(C,\cO_C(2p_1)\big)$, $H^0\big(C,\cO_C(3p_2)\big)$, $H^0\big(C,\cO_C(3p_3)\big)$. 

We notice that the three points $p_1,p_2,p_3$ are naturally ordered but not necessarily pairwise distinct, thus the $4$--tuple $(C,p_1,p_2,p_3)$ does not represent a point in the moduli space of $3$--pointed elliptic curves in general. 

Fix projective coordinates $x_0,x_1,x_2$ in $\p2$. It is well known that each abstract elliptic curve $C$ is isomorphic to a smooth cubic curve in $\p2$. Let $S\subseteq H^0\big(\p2,\cO_{\p2}(3)\big)$ the locus of polynomials  corresponding to smooth curves. The scheme
$$
Z:=\{ \ (p_1,p_2,p_3,e)\ \vert\ e(p_h)=0, \ h=1,2,3\ \}\subseteq(\p2)^{\times3}\times S
$$
is naturally fibred over $S$ and its fibre over $e$ is the product of three copies of the corresponding curve $V_+(e)$. It follows that $Z\to S$ is flat, thanks to \cite{Ha2}, Theorem III.9.9, and it has irreducible fibres, thus $Z$ is irreducible due to \cite{Ha2}, Corollary III.9.6.

Fix $e\in S$: for each $p\in \p2$ such that $e(p)=0$ we denote by $\overline{p}$ the residual intersection of the curve $V_+(e):=\{\ e=0\ \}$ with its tangent at $p$. For each polynomial $f$ we denote by $\nabla_f$ the gradient matrix. We set
 \begin{gather*}
U_{e,p}:=\{\ u\in H^0\big(\p2,\cO_{\p2}(1)\big)\ \vert u(\overline{p})=0\ \},\\
V_{e,p}:=\{\ v\in H^0\big(\p2,\cO_{\p2}(2)\big)\ \vert v(p)=v(\overline{p})=0,\ \dim\langle\nabla_e(\overline{p}),\nabla_v(\overline{p})\rangle\le1\ \}.
\end{gather*}
The sections of $U_{e,p}$ cut out on $V_+(e)$ the linear system $H^0\big(V_+(e),\cO_{V_+(e)}(2p)\big)$ residually to $\overline{p}$. Similarly, the sections of $V_{e,p}$ cut out $H^0\big(V_+(e),\cO_{V_+(e)}(3p)\big)$ residually to $p+2\overline{p}$.

Consider the variety 
\begin{align*}
{Y}:=\{ \ (p_1,p_2,p_3,e,u_1^{(1)},u_1^{(2)},v_2^{(1)},v_2^{(2)},v_3^{(1)},v_3^{(2)})\in Z\times \bP(U_{e,p_1}^{\oplus2})\times\bP(V_{e,p_2}^{\oplus2})\times\bP(V_{e,p_3}^{\oplus2})\ \}.
\end{align*}
$Y$ is endowed with a natural projection map $q\colon Y\to Z$ whose fibres are products of projective spaces of constant dimensions. By construction, it follows that $Y$ is locally trivial over $Z$. Such a map is flat because the base is irreducible and it factors via the Segre map through an embedding in
$$
Z\times \bP\left(H^0\big(\p2,\cO_{\p2}(1)\big)^{\oplus2}\otimes H^0\big(\p2,\cO_{\p2}(2)\big)^{\oplus2}\otimes H^0\big(\p2,\cO_{\p2}(2)\big)^{\oplus2}\right)
$$
followed by the projection (use again Theorem 9.9 of \cite{Ha2}), hence $Y$ is irreducible (again by \cite{Ha2}, Corollary III.9.6). The locus $Y_0\subseteq Y$ of points such that the induced map $\alpha \colon V_+(e)\to F$ is an embedding is open and non--empty, whence irreducible.

We have a natural family ${\mathcal Y}_0\subseteq Y_0\times F$ whose fibre over 
$$
(p_1,p_2,p_3,e,u_1^{(1)},u_1^{(2)},v_2^{(1)},v_2^{(2)},v_3^{(1)},v_3^{(2)})
$$
is the elliptic curve $V_+(e)$ embedded in $F$ via the sections $u_1^{(1)},u_1^{(2)},v_2^{(1)},v_2^{(2)},v_3^{(1)},v_3^{(2)}$. Since the fibres of the map ${\mathcal Y}_0\to Y_0$ induced by the projection on the first factor  are elliptic curves of degree $8$, it follows that such a family is flat. The universal property of the Hilbert scheme yields the existence of a unique morphism $Y_0\to\Hilb_{8t}^{sm}(F)$ whose image $\overline{\calH}_{c_2}$ is thus irreducible. Trivially $\overline{\calH}_{c_2}\cap\Hilb_{8t}^{sm,nd}(F)={\calH}_{c_2}$ which is thus irreducible too.

We have a natural projection $z\colon Z\to(\p2)^{\times3}$. Let $\Delta$ be the union of the diagonals of $(\p2)^{\times3}$ and let $\widetilde{Z}:=z^{-1}((\p2)^{\times3}\setminus\Delta)$. Then $\widetilde{Z}$ is an open set of a vector bundle over $(\p2)^{\times3}\setminus\Delta$ with fibre of constant dimension $7=h^0\big(\p2,\cO_{\p2}(3)\big)-3$, thus it is rational. It follows that the same is true for $Z$, hence for the open subset $Y_0\subseteq Y$. In particular ${\calH}_{c_2}$ is unirational, because it is dominated by the rational variety $Y_0$.

Again we must prove that ${\calH}_{c_2}$ is smooth of dimension $16$. Pick a point of ${\calH}_{c_2}$ corresponding to a smooth, connected, elliptic curve $C$. The cohomology of Sequence \eqref{seqIdeal} twisted by $\cE^\vee\cong\cE(-c_1)$, Lemma \ref{lObstructed} and the vanishing $h^3\big(F,\cE^\vee\big)=h^0\big(F,\cE(-2h)\big)=0$ imply that $h^2\big(F,\cI_{C\vert F}\otimes\cE\big)=0$. 

The cohomology of Sequence \eqref{seqStandard} twisted by $\cE$ and the isomorphism $\cN_{C\vert F}\cong\cE\otimes\cO_C$ yield $h^1\big(F,\cN_{C\vert F}\big)=0$. Riemann--Roch theorem on $C$ finally implies $h^0\big(F,\cN_{C\vert F}\big)=16$.

Again each component ${\calH}_{c_2}$ is globally smooth, thus the components of $\Hilb_{7t+1}^{sm,nd}(F)$ are disjoint.
\end{proof}

\section{The relative Hartshorne--Serre correspondence}
\label{sRelative}
We have thus proved the irreducibility of some particular loci in the Hilbert schemes of curves on $F$ with fixed class. We now construct on such loci flat families of vector bundles. This will allow us to define suitable maps from such loci on certain moduli space of aCM vector bundles, in order to prove their irreducibility.

Let $X$ be a smooth homogeneous variety of dimension $n\ge2$. Let $c_i\in A^i(X)$ be such that there exists a rank $2$ vector bundle $\cE_0$ over $X$ with $c_i=c_i(\cE_0)$ and $h^i\big(X,\cO_X(-c_1)\big)=0$ for $i=1,2$. 

Assume that the general section $s\in H^0\big(X,\cE_0\big)$ vanishes exactly along a subscheme $C_0$ of pure codimension $2$ whose Hilbert polynomial is $p(t)$. Thus the open locus $\Hilb_{p(t)}^{lci}(X)\subseteq\Hilb_{p(t)}(X)$ of points corresponding to locally complete intersection curves is non--empty. 

Moreover, Sequence \eqref{seqIdeal} implies that $\cN_{C_0\vert X}\cong\cE_0\otimes\cO_{C_0}$, thus $\omega_{C_0}\cong\omega_X\otimes\cO_{C_0}(c_1)$ by adjunction. Therefore the locus $\calH\subseteq \Hilb_{p(t)}^{lci}(X)$ of points representing schemes $C$ with $\omega_{C}\cong\omega_X\otimes\cO_C(c_1)$ is non--empty too.

\begin{theorem}
\label{tFamily}
With the above notation and hypothesis, then there exists a flat family ${\frak e}\colon\fE\to\calH$ of bundles of  rank $2$ on $X$ with Chern classes $c_1$, $c_2$. Moreover, $C$ is the zero--locus of a section of ${\frak e}^{-1}(C)$.
\end{theorem}
\begin{proof}
Let $\fC\subseteq \fX:=X\times\calH$ be the universal scheme, i.e. the flat family having fibre $C$ over the point corresponding to the scheme $C\in\calH$. The embedding $\fC\subseteq \fX$ is fibrewise locally complete intersection, thus it is locally complete intersection. We now construct a flat family $\fE$ of vector bundles over $\calH$ with Chern classes $c_1$ and $c_2$. To this purpose we will relativize the standard Hartshorne--Serre construction described in \cite{Ar}.

First we consider the scheme $\fX$ with the two projections $\varphi$ and $\psi$ onto $X$ and $\calH$ respectively. The morphism $\psi$ is trivially flat, thus $\cO_{\fX}( c_1):=\varphi^*\cO_X( c_1)$ is $\calH$--flat being invertible on $\fX$. It follows the flatness of the sheaf $\cO_{\fC}( c_1):=\cO_{\fC}\otimes\cO_{\fX}( c_1)$.  The exact sequence 
\begin{equation}
\label{seqRelative}
0\longrightarrow\cI_{\fC\vert\fX}\longrightarrow\cO_{\fX}\longrightarrow\cO_{\fC}\longrightarrow0
\end{equation}
yields that $\cI_{\fC\vert\fX}( c_1)$ is flat on $\calH$ too.

Now we consider the two left--exact functors $\psi_*$ and $\Hom_{\fX}\big(\cdot, \cO_{\fX}(- c_1)\big)$. The spectral sequence of the composition of these two functors satisfies
$$
E^{p,q}_2:=R^p\psi_*\left(\Ext^q_{\fX}\big(\cI_{\fC\vert\fX},\cO_{\fX}(- c_1)\big)\right),
$$
and it abuts to
$$
E^n:=R^n\left(\psi_*\Hom_{\fX}\big(\cI_{\fC\vert\fX}, \cO_{\fX}(- c_1)\big)\right).
$$

Recall that the exact sequence of low degree terms is
\begin{equation}
\label{seqLow}
0\longrightarrow E^{1,0}_2\longrightarrow E^{1}\longrightarrow E^{0,1}_2\longrightarrow E^{2,0}_2.
\end{equation}
By applying $\Hom_{\fX}\big(\cdot, \cO_{\fX}(- c_1)\big)$ to Sequence \eqref{seqRelative} we obtain
\begin{align*}
0\longrightarrow&\Hom_{\fX}\big(\cO_{\fC}, \cO_{\fX}(- c_1)\big)\longrightarrow\Hom_{\fX}\big(\cO_{\fX}, \cO_{\fX}(- c_1)\big)\longrightarrow\\
&\longrightarrow\Hom_{\fX}\big(\cI_{\fC\vert\fX}, \cO_{\fX}(- c_1)\big)\longrightarrow\Ext^1_{\fX}\big(\cO_{\fC}, \cO_{\fX}(- c_1)\big)\longrightarrow\\ 
&\longrightarrow\Ext^1_{\fX}\big(\cO_{\fX}, \cO_{\fX}(- c_1)\big)\longrightarrow\Ext^1_{\fX}\big(\cI_{\fC\vert\fX}, \cO_{\fX}(- c_1)\big)\longrightarrow\\
&\longrightarrow\Ext^2_{\fX}\big(\cO_{\fC}, \cO_{\fX}(- c_1)\big)\longrightarrow\Ext^2_{\fX}\big(\cO_{\fX}, \cO_{\fX}(- c_1)\big)\longrightarrow0.
\end{align*}

It is clear that $\Ext^i_{\fX}\big(\cO_{\fX}, \cO_{\fX}(- c_1)\big)=0$, $i\ge1$. Since $\cO_{\fC}$ is a torsion $\cO_{\fX}$--sheaf, it follows that $\Hom_{\fX}\big(\cO_{\fC}, \cO_{\fX}(- c_1)\big)=0$. Finally $\Ext^1_{\fX}\big(\cO_{\fC}, \cO_{\fX}(- c_1)\big)=0$ because the embedding $\fC\subseteq\fX$ is locally complete intersection. 

On the one hand
\begin{align*}
\Hom_{\fX}\big(\cI_{\fC\vert\fX},\cO_{\fX}(- c_1)\big)\cong \Hom_{\fX}\big(\cO_{\fX}, \cO_{\fX}(- c_1)\big)\cong \cO_{\fX}(- c_1).
\end{align*}
Since $\cO_{\fX}(- c_1)$ is flat over $\calH$ and $H^p\big(X,\cO_X(- c_1)\big)=0$, $p=1,2$, the semicontinuity theorem (see \cite{Ha2}, Corollary III.12.9) yields
$$
E^{p,0}_2=R^p\psi_* \cO_{\fX}(- c_1)\cong 0,\qquad p=1,2.
$$

On the other hand 
\begin{align*}
\Ext^1_{\fX}\big(\cI_{\fC\vert\fX}, \cO_{\fX}(- c_1)\big)\cong \Ext^2_{\fX}\big(\cO_{\fC}, \cO_{\fX}(- c_1)\big)\cong\omega_{\fC\vert \calH}\otimes\omega_{\fX\vert \calH}^{-1}\otimes\cO_{\fX}(-c_1)
\end{align*}
Since $\Pic(\fX)\cong\varphi^*\Pic(X)\oplus\psi^*\Pic(\calH)$, it follows that $\omega_{\fC\vert \calH}\otimes\omega_{\fX\vert \calH}^{-1}\otimes\cO_{\fX}(-c_1)\cong\varphi^*\cM\otimes\psi^*\calL$ for suitable $\cM\in\Pic(X)$ and $\calL\in\Pic(\calH)$. We know by adjunction formula that the restriction of $\omega_{\fC\vert \calH}$ to each fibre is $\omega_{C}\cong\omega_X\otimes\cO_C(c_1)$. We conclude that 
$$
\omega_{\fC\vert \calH}\otimes\omega_{\fX\vert \calH}^{-1}\otimes\cO_{\fX}(-c_1)\cong \psi^*\calL,
$$
whence $E^{0,1}_2\cong\psi_*\psi^*\calL\cong \calL$.

Substituting in Sequence \eqref{seqLow} we finally obtain an isomorphism
$$
\calL\cong R^1\left(\psi_*\Hom_{\fX}\big(\cI_{\fC\vert\fX}, \cO_{\fX}(- c_1)\big)\right),
$$
whence
\begin{equation}
\label{IsoGlobal}
\cO\cong R^1\left(\psi_*\Hom_{\fX}\big(\cI_{\fC\vert\fX}, \cO_{\fX}(- c_1)\otimes\psi^*{\calL}^{-1}\big)\right).
\end{equation}
Now take a sufficiently fine open cover of $\calH$ with affine open subsets $U:=\spec(A)\subseteq\calH$. We have an identification
$$
A=\ext^1_{\psi^{-1}(U)}\big(\cI_{\fC\cap\psi^{-1}(U)\vert\psi^{-1}(U)}, \cO_{\psi^{-1}(U)}(- c_1)\big).
$$
Taking the image of $1\in A$ we obtain locally on $U$ an extension of $\cI_{\fC\vert\fX}( c_1)\otimes\psi^*\calL$ with $\cO_{\fX}$.  The global isomorphism \eqref{IsoGlobal} allows us to glue together such sequences. Hence we have a global exact sequence
\begin{equation}
\label{seqGlobal}
0\longrightarrow\cO_{\fX}\longrightarrow\fE\longrightarrow\cI_{\fC\vert\fX}( c_1)\otimes\psi^*\calL\longrightarrow0.
\end{equation}
Since $\cI_{\fC\vert\fX}( c_1)$ is $\cO_{\calH}$--flat, it follows that the same is true for $\cI_{\fC\vert\fX}( c_1)\otimes\psi^*\calL$. Moreover $\cO_{\fX}$ is also flat. We conclude that the family ${\frak e}\colon\fE\to\calH$ is flat too.

Let $C\in\calH$. Recall that $C$ is locally complete intersection in $\fC$. Tensoring Sequence \eqref{seqGlobal} to $\psi^{-1}(C)$, we obtain the exact sequence \eqref{seqIdeal} with $\cE:={\frak e}^{-1}(C)$. Thus $C$ is the zero--locus of a section of $\cE$. Trivially $c_1(\cE)=c_1$.

The second Chern class $c_2(\cE)$ is the class of the zero locus  $C$ of $\cE$ in $A^2(X)$. Trivially $C$ is algebraically equivalent to $C_0$, thus they are also homologically equivalent (see \cite{Fu}, Proposition 19.1.1). Since $X$ is homogeneous, it follows that homologically equivalent cycles are rationaly equivalent (see \cite{Fu}, Example 19.1.11). Thus the class of each fibre $C$ of $\fC$ inside $A^2(X)$ is constantly $c_2$. 
\end{proof}

In the next sections we will apply the above result for constructing suitable families of bundles on $X=F$ where $\calH$ is one of the components of the schemes $\Hilb_{t+1}(F)$, $\Hilb_{7t+1}^{sm,nd}(F)$, $\Hilb_{8t}^{sm,nd}(F)$ described in the previous section.

\section{Moduli spaces of Ulrich bundles}
\label{sModuliUlrich}
We are interested in initialized rank $2$ aCM bundles on $F$ with $c_1$ either $2h$ or $h_1+2h_2+3h_3$: the representative $c_2$ are uniquely determined (see the Proposition \ref{pList}). 

Let $\cE$ be a vector bundle. Then $\End(\cE)\cong H^0\big(F,\cE\otimes{\cE}^\vee\big)$ has dimension at least $1$ and we call $\cE$ simple if such a minimum is attained. If $\cE$ is stable, then it is also simple (see \cite{H--L}, Corollary 1.2.8). A similar property holds when $\cE$ is an Ulrich bundle on $F$, even without the stability property.

\begin{proposition}
\label{pSimple}
If $\cE$ is an indecomposable Ulrich bundle of rank $2$ on $F$, then it is simple.
\end{proposition}
\begin{proof}
Due to the discussion above we can restrict our attention to strictly $\mu$--semistable Ulrich bundles $\cE$. We know that the bundle  $\cE$ fits into Sequence \eqref{seqUnstable} where $\calL\cong\cO_F(c_1-2h_2-h_3)$.

Taking the cohomology of Sequence \eqref{seqUnstable} tensorized by ${\cE}^\vee$ we obtain 
\begin{equation}
\label{Bound}
1\le h^0\big(F,\cE\otimes{\cE}^\vee\big)\le h^0\big(F,{\cE}^\vee\otimes\calL\big)+h^0\big(F,{\cE}^\vee(2h_2+h_3)\big).
\end{equation}
By applying the functor $\mathrm{Hom}_F\big(\cO_F(2h_2+h_3),\cdot\big)$ to Sequence \eqref{seqUnstable} and taking into account that $\cE$ is indecomposable, we deduce that 
$$
h^0\big(F,{\cE}^\vee\otimes\calL\big)=h^0\big(F,\cE(-2h_2-h_3)\big)=h^0\big(F,\calL(-2h_2-h_3)\big)=0. 
$$
Taking the cohomology of  Sequence \eqref{seqUnstable}  twisted by $\calL^{-1}\cong \cO_F(2h_2+h_3-c_1)$ we obtain 
$$
h^0\big(F,{\cE}^\vee(2h_2+h_3)\big)=h^0\big(F,\cE(-h_1-2h_3)\big)=h^0\big(F,\cO_F)=1.
$$
Thus Inequalities \eqref{Bound} yields $\dim_k\left(\End(\cE)\right)=h^0\big(F,\cE\otimes{\cE}^\vee\big)=1$.
\end{proof}

We now deal with the irreducibility and the dimension of the moduli spaces constructed above. We start with the case $c_1=h_1+2h_2+3h_3$: recall that in this case $c_2$ is either $4h_2h_3+h_1h_3+2h_1h_2$ or $3h_2h_3+3h_1h_3+h_1h_2$. The first result is the following lemma reverting Proposition \ref{pStable} in this particular case.

\begin{lemma} 
\label{lExtension}
Let $\cE$ be an initialized rank $2$ aCM bundle on $F$ with $c_1=h_1+2h_2+3h_3$ and $c_2=4h_2h_3+h_1h_3+2h_1h_2$. Then $\cE$ fits into an exact sequence of the form
$$
0\longrightarrow\cO_F(h_1+2h_3)\longrightarrow \cE\longrightarrow\cO_F (2h_2+h_3)\longrightarrow0.
$$
In particular there exists a family with base $\p3$ parameterizing such bundles.
\end{lemma}
\begin{proof}
Let $\cE$ be as in the statement: Riemann--Roch theorem yields $\chi(\cE(-h_1-2h_3))=1$. 
If $C\subseteq F$ is the zero locus of a general section of $\cE$, then $C$ is a rational normal curve of degree $7$.
Taking the cohomology of  Sequences \eqref{seqIdeal} and \eqref{seqStandard} respectively twisted by $\cO_F(-h_1-2h_3)$ and $\cO_F(2h_2+h_3)$ we obtain
\begin{align*}
h^2\big(F,\cE(-h_1-2h_3)\big)&=h^2\big(F,\cI_{C\vert F}(2h_2+h_3)\big)=\\
&=h^1\big(F,\cO_C(2h_2+h_3)\big)=h^1\big(\p1,\cO_{\p1}(4)\big)=0.
\end{align*}
Hence $h^0\big(F,\cE(-h_1-2h_3)\big)=h^0\big(F,{\cE}^\vee(2h_2+h_3)\big)\ne0$. 

Let $\sigma\in H^0\big(F,\cE^\vee(2h_2+h_3)\big)$ and set $(\sigma)_0=D\cup E$, where $D\in \vert a_1h_1+a_2h_2+a_3h_3\vert$ is an effective divisor (i.e. $a_i\ge0$, $i=1,2,3$) and $E$ is either empty or has pure dimension $1$. Thus $E=(s)_0$ where $s\in H^0\big(F,\cE^\vee(2h_2+h_3-D)\big)$. Twisting by $\cO_F(D)$ the Koszul complex of $s$ we obtain
\begin{equation}
\label{Extension} 
0\longrightarrow \cO_F(D)\longrightarrow {\cE}^\vee(2h_2+h_3)\longrightarrow \cI_{E\vert F}(-h_1+2h_2-h_3-D)\longrightarrow 0.
\end{equation}
Twisting Sequence \eqref{Extension} by $\cO_F(h_1+2h_3)$ we obtain
$$
0\longrightarrow \cO_F(D+h_1+2h_3)\longrightarrow \cE\longrightarrow \cI_{E\vert F}(2h_2+h_3-D)\longrightarrow 0.
$$
We know that $h^0\big(F,\cO_F(D-h_2+h_3)\big)\le h^0\big(F,\cE(-h)\big)=0$, thus $a_2=0$. We also know that $\cE$ is also globally generated, thus the same is true for $\cI_{E\vert F}(2h_2+h_3-D)$. In particular $0<h^0\big(F,\cI_{E\vert F}(2h_2+h_3-D)\big)\le h^0\big(F,\cO_F(2h_2+h_3-D)\big)$, whence we infer $a_1=0$ and $a_3\le1$. Taking the cohomology of Sequence \eqref{Extension} twisted by $\cO_F(-2h_2-h_3)$ we obtain 
$$
0\longrightarrow \cO_F(D-2h_2-h_3)\longrightarrow {\cE}^\vee\longrightarrow \cI_{E\vert F}(-h_1-2h_3-D)\longrightarrow 0.
$$
Since $D\in\vert a_3 h_3\vert$, it follows that $h^0\big(F,\cI_{E\vert F}(-h_1-2h_3-D)\big)\le h^0\big(F,\cO_F(-h_1-2h_3-D)\big)=0$. Moreover ${\cE}^\vee$ is aCM. Taking the cohomology of the above sequence we thus obtain $h^1\big(F,\cO_F(D-2h_2-h_3)\big)=0$, whence $a_3=0$. 

We conclude that $D=0$ in Sequence \eqref{Extension}. Since $c_2({\cE}^\vee(2h_2+h_3))=0$ we deduce that $E=\emptyset$. In particular $\cI_{E\vert F}\cong\cO_F$ in Sequence \eqref{Extension}. Twisting such a sequence by $\cO_F(h_1+2h_3)$ we finally prove the existence of the extension we were asking for. Such extension are parameterized by $\p3$, because $\dim\big(\ext^1_F\big(\cO_F (2h_2+h_3),\cO_F(h_1+2h_3)\big)\big)=4$.
\end{proof}

Lemma \ref{lExtension} yields that a Jordan--H\"older filtration of $\cE$ is $0\subset\cO_F(h_1+2h_3)\subset \cE$: indeed $\cO_F(h_1+2h_3)$ and $\cO_F (2h_2+h_3)$, being invertible, are stable with reduced Hilbert polynomial $p(t)=6{t+3\choose3}$. It follows that
$$
\gr(\cE)=\cO_F(h_1+2h_3)\oplus \cE/\cO_F(h_1+2h_3)\cong \cO_F(h_1+2h_3)\oplus\cO_F (2h_2+h_3).
$$
Thus we have just proved the following result.

\begin{proposition}
\label{pPoint}
The moduli space
$$
\cM_F^{ss, U}(2;h_1+2h_2+3h_3,4h_2h_3+h_1h_3+2h_1h_2)
$$
reduces to a single point.
\end{proposition}

In all the other cases we claim that there always exist stable Ulrich bundles. Thanks to Proposition \ref{pStable} this is obvious when $c_2$ is either $2h_2h_3+3h_1h_3+3h_1h_2$ or $3h_2h_3+3h_1h_3+h_1h_2$. In particular we can prove the following results.

\begin{proposition}
\label{pModuliRationalElliptic}
The moduli spaces
\begin{gather*}
\cM_F^{ss,U}(2;h_1+2h_2+3h_3,3h_2h_3+3h_1h_3+h_1h_2),\\
\cM_F^{ss,U}(2;2h,2h_2h_3+3h_1h_3+3h_1h_2),
\end{gather*}
are irreducible, smooth of respective dimensions $3$ and $5$. 

They coincide with the loci of stable bundles.
\end{proposition}
\begin{proof}
The existence of such moduli spaces has been already stated in Proposition \ref{pModUlrich}. We want to prove their irreducibility. We know that the locus $\calH\subseteq\Hilb_{7t+1}^{sm,nd}(F)$ corresponding to smooth, connected, non--degenerate curves with class $c_2=3h_2h_3+3h_1h_3+h_1h_2$ is irreducible, smooth and unirational (see Proposition \ref{pRational}). Moreover $h^i\big(F,\cO_F(-c_1)\big)=0$, $i=1,2$, where $c_1=h_1+2h_2+3h_3$.

It follows the existence of a flat family $\fE\to\calH$ of vector bundles of rank $2$ with Chern classes $c_1$ and $c_2$ (see Theorem \ref{tFamily}). If the bundle $\cE$ is a fibre of such a family, then it fits into the  exact sequence \eqref{seqIdeal} with $c_1=h_1+2h_2+3h_3$ and $C\subseteq F$ a rational normal curve. Trivially $\cE$ is initialized. Thanks to \cite{C--F--M}, Section 7, we know that $\cE$ is also aCM.

Thus the universal property of $\cM_F^{ss}(2;h_1+2h_2+3h_3,3h_2h_3+3h_1h_3+h_1h_2)$ yields the existence of a well--defined morphism
$$
m\colon\calH\to \cM_F^{ss,U}(2;h_1+2h_2+3h_3,3h_2h_3+3h_1h_3+h_1h_2).
$$
The morphism $m$ is surjective thanks to Proposition \ref{pList} (4): indeed each initialized, aCM bundle with $c_1=h_1+2h_2+3h_3$ and $c_2=3h_2h_3+3h_1h_3+h_1h_2$ has a rational normal curve as zero--locus of its general section, thus it appears as a fibre of the family defined in Theorem \ref{tFamily}. We conclude that $\cM_F^{ss,U}(2;h_1+2h_2+3h_3,3h_2h_3+3h_1h_3+h_1h_2)$ is irreducible. 

Each point in $\cM_F^{ss,U}(2;h_1+2h_2+3h_3,3h_2h_3+3h_1h_3+h_1h_2)$ represents a stable bundle. Corollary 4.5.2 of \cite{H--L} yields that it is a smooth point, thanks to the vanishing $h^2\big(F,{\cE}^\vee\otimes\cE\big)=0$ proved in Lemma \ref{lObstructed}. Moreover we have
$$
\dim\big(\cM_F^{ss,U}(2;h_1+2h_2+3h_3,3h_2h_3+3h_1h_3+h_1h_2)\big)=h^1\big(F,{\cE}^\vee\otimes\cE\big).
$$
Thanks to Proposition \ref{pSimple} and Lemma \ref{lObstructed} we have $h^1\big(F,{\cE}^\vee\otimes\cE\big)=1-\chi({\cE}^\vee\otimes\cE)$. It is easy to check that $c_1({\cE}^\vee\otimes\cE)=0$, $c_2({\cE}^\vee\otimes\cE)=4c_2-c_1^2$, $c_3({\cE}^\vee\otimes\cE)=0$: thus Riemann--Roch theorem for ${\cE}^\vee\otimes\cE$ yields $\chi({\cE}^\vee\otimes\cE)=4-4hc_2+hc_1^2$, hence $h^1\big(F,{\cE}^\vee\otimes\cE\big)=3$.

The second case can be handled similarly.
\end{proof}

What can be said in the case $c_2=2h_2h_3+2h_1h_3+4h_1h_2$? Take a strictly semistable Ulrich bundle $\cE$. It has a Jordan--H\"older filtration $0\subset\calL\subset \cE$ where $\calL$ is either $\cO_F(2h_1+h_3)$ or $\cO_F(2h_2+h_3)$ (see the proof of Proposition \ref{pStable}), hence $\cE/\calL$ is, respectively, either $\cO_F(2h_2+h_3)$ or $\cO_F(2h_1+h_3)$. In particular 
$$
\gr(\cE)\cong\cO_F(2h_1+h_3)\oplus\cO_F(2h_2+h_3)
$$
Hence there is only one point in $\cM_F^{ss,U}(2;2h,2h_2h_3+2h_1h_3+4h_1h_2)$ representing the $S$--equivalence class of all the strictly semistable bundles.

\begin{proposition}
The moduli space
\begin{gather*}
\cM_F^{ss,U}(2;2h,2h_2h_3+2h_1h_3+4h_1h_2),
\end{gather*}
is irreducible of dimension $5$. 

The locus $\cM_F^{s,U}(2;2h,2h_2h_3+2h_1h_3+4h_1h_2)$ is irreducible, smooth and its complement consists in exactly one point.
\end{proposition}
\begin{proof}
We still have a surjective morphism
$$
m\colon\calH\to \cM_F^{ss,U}(2;2h,2h_2h_3+2h_1h_3+4h_1h_2)
$$
where $\calH$ is the component of the Hilbert scheme $\Hilb_{8t}^{sm,nd}(F)$ whose points correspond to curves in the class $2h_2h_3+2h_1h_3+4h_1h_2$. Proposition \ref{pElliptic} guarantees that $\calH$ has dimension $16$. Consider the fibre $X$ over the unique point corresponding to the  $S$--equivalence class of all the strictly semistable bundles. 

Recall that strictly semistable bundles are parameterized by the variety $B:=\bP\big(\ext^1_F\big(\cO_F (2h_2+h_3),\cO_F(2h_1+h_3)\big)\big)\cong\p2$ (see Proposition \ref{pStable}). We have a map over $B$, whose fibre over a point corresponding to the extension 
$$
0\longrightarrow\cO_F(2h_1+h_3)\longrightarrow\cE\longrightarrow\cO_F (2h_2+h_3)\longrightarrow0
$$
is dominated by $\bP(H^0\big(F,\cE\big))$. Since $\cE$ is Ulrich, it follows that $h^0\big(F,\cE\big)=12$, hence $\dim(X)\le13$. In particular there exist points in $\calH$ which are mapped on points representing stable bundles.

Repeating almost verbatim the argument used in the proof of Proposition \ref{pModuliRationalElliptic} we are able to complete the proof of the statement.
\end{proof}

\section{Moduli spaces of non--Ulrich bundles}
\label{sModuliLine}
In this section we will examine the moduli spaces of non--Ulrich semistable bundles $\cE$. 
Thus we start to examine the case of initialized rank $2$ aCM bundles on $F$ with $c_1=h_3$: Proposition \ref{pList} shows that the representative $c_2$ can be assumed to be $h_2h_3$. Recall that such bundles are $\mu$--stable, hence stable, thus their moduli space 
$$
\cM_F^{s, aCM}(2;h_3,h_2h_3),
$$
parameterizing their isomorphisms classes (see \cite{C--H2}), exists. Thanks to  Proposition \ref{pHilbertLine} we know that the locus $\calH\subseteq\Hilb_{t+1}(F)$ corresponding to lines with class $c_2=h_2h_3$ is isomorphic to $\p1\times\p1$, hence rational and smooth.  Moreover $h^i\big(F,\cO_F(-c_1)\big)=0$, $i=1,2$, where $c_1=h_3$.

Thus there exists a flat family $\fE\to\calH$ of vector bundles of rank $2$ with Chern classes $c_1$ and $c_2$ (see Theorem \ref{tFamily}). If the bundle $\cE$ is a fibre of such a family, then it fits into the exact sequence \eqref{seqIdeal} where $c_1=h_3$ and $C\subseteq F$ is a line. Thanks to \cite{C--F--M}, Section 8.2, we know that $\cE$ is also aCM.

Thus the universal property of $\cM_F^{s,aCM}(2;h_3,h_2h_3)$ yields the existence of a well--defined morphism $m\colon\p1\times\p1\to \cM_F^{s,aCM}(2;h_3,h_2h_3)$. The morphism $m$ is surjective thanks to assertion (1) of Proposition \ref{pList}: indeed we have that each initialized, aCM bundle with $c_1=h_3$ and $c_2=h_2h_3$ has a line as zero--locus of its general section, thus it appears as a fibre of the family defined in Theorem \ref{tFamily}. It follows that such a moduli space is irreducible.

Let $\cE$ be one of the bundles we wish to deal with. Being stable such a bundle is simple  (see \cite{H--L}, Corolary 1.2.8), hence $h^0\big(F,\cE\otimes{\cE}^\vee\big)=1$. Let $C$ be the zero--locus of a general section of $\cE$: we already know that $C$ is a line whose class in $A^2(F)$ is $h_2h_3$. Tensoring Sequence \eqref{seqIdeal} by $\cO_F(-h_3)$ and taking its cohomology we obtain $h^i\big(F,\cE(-h_3)\big)=h^i\big(F,\cI_{C\vert F}\big)$. Since $C$ is a line and the embedding $F\subseteq\p7$ is aCM, it is easy to check that $h^i\big(F,\cI_{C\vert F}\big)=0$, $i=0,1,2,3$.

Again the cohomology of the same sequence tensorized by $\cE(-h_3)$ yields $h^i\big(F,\cE\otimes{\cE}^\vee\big)=h^i\big(F,\cE\otimes\cI_{C\vert F}\big)$, $i=0,1,2,3$. In order to compute these last dimensions we can take the cohomology of Sequence \eqref{seqStandard} tensorized by $\cE$. We already know that $h^0\big(F,\cE\otimes\cI_{C\vert F}\big)=h^0\big(F,\cE\otimes{\cE}^\vee\big)=1$. Moreover, since $\cE$ is aCM we obtain
\begin{gather*}
h^1\big(F,\cE\otimes{\cE}^\vee\big)=h^1\big(F,\cE\otimes\cI_{C\vert F}\big)=h^0\big(C,\cE\otimes\cO_{C}\big)-h^0\big(F,\cE\big)+1,\\
h^2\big(F,\cE\otimes{\cE}^\vee\big)=h^2\big(F,\cE\otimes\cI_{C\vert F}\big)=h^1\big(C,\cE\otimes\cO_{C}\big).
\end{gather*}

We know that $\cE\otimes\cO_{C}\cong\cN_{C\vert F}$. The general theory of del Pezzo threefolds implies that $\cN_{C\vert F}$ is either $\cO_{\p1}^{\oplus2}$, or $\cO_{\p1}(-1)\oplus \cO_{\p1}(1)$ (Lemma 3.3.4 of \cite{I--P}). It follows that $h^2\big(F,\cE\otimes{\cE}^\vee\big)=0$ and $h^1\big(F,\cE\otimes{\cE}^\vee\big)=3-h^0\big(F,\cE\big)$.

Again the cohomology of Sequence \eqref{seqIdeal} gives $h^0\big(F,\cE\big)=h^0\big(F,\cO_F\big)+h^0\big(F,\cI_{C\vert F}(h_3)\big)$ and $h^1\big(F,\cI_{C\vert F}(h_3)\big)=0$, because $\cE$ is aCM. Thus the cohomology of Sequence \eqref{seqStandard} yields $h^0\big(F,\cI_{C\vert F}(h_3)\big)=h^0\big(F,\cO_F(h_3)\big)-h^0\big(C,\cO_C(h_3)\big)=1$, because $h_3C=0$, whence $\cO_F(h_3)\cong\cO_F$. It follows that $h^1\big(F,\cE\otimes{\cE}^\vee\big)=1$. We have thus proved that $\cM_F^{s, aCM}(2;h_3,h_2h_3)$ is irreducible and smooth of dimensions $1$. In particular it is rational, thanks to L\"uroth theorem: we will actually check that it is isomorphic to $\p1$.

Now consider a bundle $\cE$ arising from an extension of the form
$$
0\longrightarrow \cO_F(-h_2+h_3)\longrightarrow \cE\longrightarrow \cO_F(h_2)\longrightarrow 0.
$$
Clearly such an $\cE$ is initialized, aCM, with $c_1(\cE)=h_3$ and $c_2(\cE)=h_2h_3$. It is easy to check via a Chern classes computation that if $\cE$ is decomposable then it is $\cO_F(-h_2+h_3)\oplus\cO_F(h_2)$. If this is the case, then the above sequence splits, because there are no non--zero morphisms $\cO_F(-h_2+h_3)\to\cO_F(h_2)$. 

Thus, if we only consider the bundles arising from non--trivial extensions, they are indecomposable. Notice that  $h^1\big(F,\cO_F(-2h_2+h_3)\big)=2$, thus we have a family of non--isomorphic bundles with base $\p1$. In particular $\p1\subseteq\cM_F^{s, aCM}(2;h_3,h_2h_3)$, thus equality must hold.

The duality morphism defined by $\cE\mapsto {\cE}^\vee(h)$ is an isomorphism. Thus the same conclusions hold for bundles with $c_1=2h_1+2h_2+h_3$ whose representative $c_2$ is $2h_2h_3 +h_1h_3 +2h_1h_2$, hence also the moduli space 
$$
\cM_F^{s, aCM}(2;2h_1+2h_2+h_3,2h_2h_3 +h_1h_3 +2h_1h_2)
$$
exists. 

We can summarize the above computations in the following statement.

\begin{proposition}
\label{pRationalLine}
The moduli spaces
$$
\cM_F^{s, aCM}(2;h_3,h_2h_3),\qquad \cM_F^{s, aCM}(2;2h_1+2h_2+h_3,2h_2h_3 +h_1h_3 +2h_1h_2)
$$
are isomorphic to $\p1$. 
\end{proposition}

Now consider the case of an initialized, indecomposable, aCM bundle $\cE$ of rank $2$ with $c_1(\cE)=0$ and $c_2(\cE)=h_2h_3$. Such bundles are never semistable, though $\mu$--semistable (see Proposition \ref{pLine}). It is well--known that it is possible to construct the moduli space $\cM_F^{\mu ss}(2;0,h_2h_3)$ parameterizing $S$-equivalence classes of $\mu$--semistable rank $2$ vector bundles with fixed Chern classes $c_1=0$ and $c_2=h_2h_3$ (see Section 5 of \cite{G--T}): we will denote by $\cM_F^{\mu ss, aCM}(2;0,h_2h_3)$ the locus of aCM ones.

Looking at the cohomology of Sequence \eqref{seqIdeal}, we know that $h^0\big(F,\cE\big)=1$ and each non--zero section of $\cE$ vanishes along the same line $C$ which is the complete intersection inside $F$ of two divisors $D_2\in\vert h_2\vert$ and $D_3:=\vert h_3\vert$.

Sequence \eqref{seqIdeal} implies the existence of a filtration $0\subseteq\cO_F\subseteq\cE$ with $\cE/\cO_F\cong\cI_{C\vert F}$. We claim that such a filtration is actually the Jordan--H\"older filtration of $\cE$ with respect to the $\mu$--semistability notion. 
On the one hand we know that $\mu(\cE)=0$ and it is trivial that $\mu(\cO_F)=0$, hence $\mu(\cE/\cO_F)=0$. On the other hand the $\mu$--stability of  $\cO_F$ and $\cE/\cO_F\cong\cI_{C\vert F}$ is well--known. In particular
$$
\gr(\cE)= \cO_F\oplus\cE/\cO_F\cong\cO_F\oplus\cI_{C\vert F}.
$$

Let $\cE'$ be another bundle with the same properties and let $C'$ be the zero--locus of any non--zero section of $\cE'$: we can write $C'=D_2'\cap D_3'$ with $D_2'\in\vert h_2\vert$ and $D_3':=\vert h_3\vert$. Let us look at the group $G:=\PGL_2\times\PGL_2$ as the  subgroup of the automorphism group of $F$ acting on the second and third factor of $F\cong\p1\times\p1\times\p1$. The divisors $D_i$ and $D_i'$ are inverse images of suitable points with respect to the projections $\pi_i$. it follows the existence of an element of $G$ transforming $C$ into $C'$. Its induced action on $\cO_F$ gives an isomorphism $\cI_{C\vert F}\cong\cI_{C'\vert F}$ as $\cO_F$--modules. It follows that $\gr(\cE)\cong\gr(\cE')$, i.e. all the bundles we are interested in belong to the same $S$--equivalence class with respect to $\mu$--semistability. 

\begin{proposition}
\label{pMuRationalLine}
The moduli space
$$
\cM_F^{\mu ss, aCM}(2;0,h_2h_3)
$$
reduces to a single point.
\end{proposition}

\section{(Uni)rationality of the constructed moduli spaces}
\label{sRational}
In this section we will finally discuss about the (uni)rationality of the moduli spaces constructed in the previous sections.

In Proposition \ref{pRationalLine} we proved that $\cM_F^{s, aCM}(2;h_3,h_2h_3)\cong \cM_F^{s, aCM}(2;2h_1+2h_2+h_3,2h_2h_3 +h_1h_3 +2h_1h_2)\cong\p1$. In particular they are rational.

Let us examine the moduli spaces $\cM_F^{ss,U}(2;h_1+2h_2+3h_3,3h_2h_3+3h_1h_3+h_1h_2)$ and $\cM_F^{ss,U}(2;2h,2h_2h_3+3h_1h_3+3h_1h_2)$. Thanks to the construction described in the previous section, such spaces are dominated by unirational varieties (see Propositions \ref{pRational} and \ref{pElliptic}). We conclude that such two spaces are unirational.

We consider the moduli space $\cM_F^{ss,U}(2;2h,2h_2h_3+2h_1h_3+4h_1h_2)$. The same argument used above implies it is unirational: in what follows we will show it is actually rational.

We consider the projection $\pi\colon F\to Q:=\p1\times\p1$ onto the two first factors and the moduli space $\cM_Q^{\mu s}(2;0,2)$ of $\mu$--stable rank $2$ vector bundles $\cF$ on $Q$ with Chern classes $c_1(\cF)=0$ and $c_2(\cF)=2$ which has been studied in \cite{Sob}: in particular it is irreducible, of dimension $5$ and rational. 

We denote by $C$ a general plane section of  $Q$ embedded in $\p3$ via the Segre map (or, in other words, a general divisor of bidegree $(1,1)$ on $Q$): notice that $C$ is a smooth conic, thus $C\cong\p1$. Each bundle $\cF$ representing a point in $\cM_Q^{\mu s}(2;0,2)$ is stable and normalized. It follows that $h^0\big(Q,\cF(tC)\big)=0$, $t\le -1$. For each such an $\cF$, we define $e(\cF):=\pi^*\cF\otimes\cO_F(h)$.

\begin{lemma}
If $\cF\in\cM_Q^{\mu s}(2;0,2)$, then $e(\cF)\in\cM_F^{ss,U}(2;2h,2h_2h_3+2h_1h_3+4h_1h_2)$.
\end{lemma}
\begin{proof}
An easy Chern class computation shows that $c_1(e(\cF))=2h$ and $c_2(e(\cF))=2h_2h_3+2h_1h_3+4h_1h_2$. Thus, if we show that $e(\cF)$ is aCM we are done, because aCM bundles on $F$ with those Chern classes are automatically Ulrich and semistable. We will compute the intermediate cohomology of $e(\cF)$ by using K\"unneth formula and projection formula. We have 
\begin{equation}
\label{Kunneth}
h^i\big(F,e(\cF)(th)\big)=\sum_{j=0}^ih^j\big(Q,\cF(tC)\big)h^{i-j}\big(\p1,\cO_{\p1}(t)\big).
\end{equation}
Since $c_1(e(\cF))=2h$, it follows from Serre's duality that it suffices to consider only the case $i=1$. If $t=-1$ then it becomes 
\begin{align*}
h^1\big(F,e(\cF)(-h)\big)&=h^0\big(Q,\cF(-C)\big)h^1\big(\p1,\cO_{\p1}(-1)\big)+\\
&+h^1\big(Q,\cF(-C)\big)h^0\big(\p1,\cO_{\p1}(-1)\big)=0.
\end{align*}

Let now examine the case $t\ne-1$. Let  $\cF_C:=\cF\otimes\cO_C$: the isomorphism $C\cong\p1$ implies that $\cF_C\cong\cO_{\p1}(-a)\oplus\cO_{\p1}(a)$ for some $a\ge 0$. Assume $a\ge 1$: thus the Harder--Narasimhan filtration of $\cF_C$ is $0\subseteq \cO_{\p1}(a)\subseteq \cF_C$ and $\cF_C/\cO_{\p1}(a)\cong\cO_{\p1}(-a)$, thus $0<2a\le 2$ thanks to Theorem 4.6 of \cite{Mar}. We conclude that  $0\le a\le1$. In both the cases $h^1\big(C,\cF_C(-C)\big)=h^0\big(C,\cF_C\big)=2$. For each $t\in\bZ$ consider the restriction sequence
\begin{equation}
\label{seqRestr}
0\longrightarrow\cF((t-1)C)\longrightarrow\cF(tC)\longrightarrow\cF_C(tC)\longrightarrow 0.
\end{equation}

The bundle $\cF$ is stable on $Q$, hence Serre's duality implies $h^2\big(Q,\cF(-C)\big)=h^0\big(Q,\cF(-C)\big)=0$ and $h^2\big(Q,\cF(-2C)\big)=h^0\big(Q,\cF\big)=0$. 

From the former vanishing, Riemann--Roch theorem on $Q$ and Serre's duality, we deduce $h^1\big(Q,\cF(-C)\big)=-\chi(\cF(-C))=2$. 

The cohomology of Sequece \eqref{seqRestr} with $t=-1$ gives
$$
0\longrightarrow H^1\big(Q,\cF(-2C)\big)\longrightarrow H^1\big(Q,\cF(-C)\big)\longrightarrow  H^1\big(Q,\cF_C(-C)\big)\longrightarrow 0.
$$
Thus the latter vanishing, the equalities $h^1\big(Q,\cF(-C)\big)=h^1\big(C,\cF_C(-C)\big)=2$ and Serre's duality yield $h^1\big(Q,\cF\big)=h^1\big(Q,\cF(-2C)\big)=0$.

Since $h^0\big(C,\cF_C(tC)\big)=0$ for each $t\le -2$, it follows that the map $H^1\big(Q,\cF((t-1)C)\big)\to H^1\big(Q,\cF(tC)\big)$ is injective, thus $h^1\big(Q,\cF(tC)\big)=0$ in the same range.

Since $h^1\big(C,\cF_C(tC)\big)=0$ for each $t\ge0$, it follows that the map $H^1\big(Q,\cF((t-1)C)\big)\to H^1\big(Q,\cF(tC)\big)$ is surjective, thus again $h^1\big(Q,\cF(tC)\big)=0$ in the same range.

Equality \eqref{Kunneth} for $i=1$ and $t\ne -1$ becomes
$$
h^1\big(F,e(\cF)(th)\big)=h^0\big(Q,\cF(tC)\big)h^1\big(\p1,\cO_{\p1}(t)\big).
$$
On the one hand we already know that $h^0\big(Q,\cF(tC)\big)=0$, $t\le -2$. On the other hand $h^1\big(\p1,\cO_{\p1}(t)\big)=0$, $t\ge0$. 

We conclude that $h^1\big(F,e(\cF)(th)\big)=0$ for each $t\in\bZ$.
\end{proof}

The above proposition gives the existence of a well--defined map
$$
e\colon \cM_Q^{\mu s}(0,2)\to\cM_F^{ss,U}(2;2h,2h_2h_3+2h_1h_3+4h_1h_2).
$$
The morphism $\pi$ has a section $\sigma\colon Q\to F$. In particular $\sigma^*\pi^*=(\pi\sigma)^*$ is the identity, whence we deduce the injectivity of $e$. Since both the spaces has dimension $5$, we conclude that $e$ is obviously dominant.

\begin{proposition}
The moduli space $\cM_F^{ss,U}(2;2h,2h_2h_3+2h_1h_3+4h_1h_2)$ is rational.
\end{proposition}
\begin{proof}
Since $e$ is injective and
$$
\dim(\cM_Q^{\mu s}(2;0,2))=\dim(\cM_F^{ss,U}(2;2h,2h_2h_3+2h_1h_3+4h_1h_2))=5.
$$
it follows that $e$ is birational. Since the rationality of $\cM_Q^{\mu s}(2;0,2)$ has been proved in \cite{Sob}, we deduce that $\cM_F^{ss,U}(2;2h,2h_2h_3+2h_1h_3+4h_1h_2)$ is rational too.
\end{proof}

\begin{remark}  
In  \cite{LP} J. Le Potier analyzes the restriction to the quadric $Q$ of the null correlation bundles $\mathcal N$. We denote by $\cM_{\p3}^{ss}(2;0,1)$ the moduli space of normalized semistable bundles of rank $2$ on $\p3$ with $c_2=1$. Let $\cM_{\p3}^{ss,0}(2;0,1)$ be the open subset of $\cM_{\p3}^{ss}(2;0,1)$ consisting of all bundles $\mathcal N$ such that $\mathcal N\otimes\cO_Q$ is stable on $Q$. 

The restriction gives an \'etale quasi--finite morphism from $\cM_{\p3}^{ss,0}(2;0,1)$ onto an
open proper subset $\mathcal U \subset\cM_Q^{\mu s}(2;0,2)$. The generic bundle $\cE$ of $\mathcal U$ has a twin pair (a Tjurin pair) of null correlation bundles restricting to it, while there are bundles $\cE$ in $\mathcal U$ with a unique null correlation bundle
restricting to it. 

Example 4.2  of \cite{M--R} shows the existence of a bundle in  $\cM_Q^{\mu s}(2;0,2)$ but not in $\mathcal U$, i.e. a stable bundle which is not the restriction of a null correlation bundle. The pull--back of this bundle gives an example of Ulrich bundle $\cE$ on $F$ with $c_1(\cE)=2h$ and $c_2(\cE)=2h_2h_3+2h_1h_3+4h_1h_2$ which is not related, under this construction to an instanton bundle on $\p3$. By the way, if $\cE$ is a generic bundle on $\cM_F^{ss,U}(2;2h,2h_2h_3+2h_1h_3+4h_1h_2)$ then $\cE(-h)$ can be obtained from a pair of null correlation bundles, hence $\cE(-h)$ is the homology of the monad
$$
0\longrightarrow\cO_F(-h_1-h_2)\longrightarrow\cO_F^{\oplus4}\longrightarrow\cO_F(h_1+h_2)\longrightarrow 0.
$$
   \end{remark}

\bigskip
\noindent
Gianfranco Casnati,\\
Dipartimento di Scienze Matematiche, Politecnico di Torino,\\
c.so Duca degli Abruzzi 24,\\
10129 Torino, Italy\\
e-mail: {\tt gianfranco.casnati@polito.it}

\bigskip
\noindent
Daniele Faenzi,\\
UFR des Sciences et des Techniques, Universit\'e de Bourgogne,\\
9 avenue Alain Savary -- BP 47870\\
21078 DIJON CEDEX, France\\
e-mail: {\tt daniele.faenzi@univ-pau.fr}

\bigskip
\noindent
Francesco Malaspina,\\
Dipartimento di Scienze Matematiche, Politecnico di Torino,\\
c.so Duca degli Abruzzi 24,\\
10129 Torino, Italy\\
e-mail: {\tt francesco.malaspina@polito.it}

\end{document}